\documentclass[12pt,a4paper]{article}
\usepackage[T1]{fontenc}
\usepackage[utf8]{inputenc}
\usepackage[english]{babel}
\usepackage{amsmath}
\usepackage{cite}
\usepackage{amsfonts}
\usepackage{amssymb}
\usepackage{amsthm}
\usepackage{mathtools}
\usepackage{marginnote}
\usepackage{verbatim}
\mathtoolsset{showonlyrefs=true}
\usepackage[colorlinks=true,allcolors=blue]{hyperref}
\usepackage[left=2.5cm,right=2.5cm,top=2.5cm,bottom=2.5cm]{geometry}
\usepackage{float}
\usepackage{authblk}
\usepackage{enumerate}
\usepackage{physics}
\usepackage{subcaption}
\usepackage{mhchem}
\newtheorem{thm}{Theorem}[section]

\newtheorem{prop}[thm]{Proposition}

\newtheorem{rem}[thm]{Remark}

\usepackage{tikz}
\usetikzlibrary{graphs}
\usetikzlibrary{shadows}
\usetikzlibrary{fit}
\definecolor{myblue}{RGB}{240,240,255}
\definecolor{myshadow}{RGB}{200,200,200}
\definecolor{mygray}{RGB}{230,230,230}
\tikzstyle{complex}=[shape=rectangle, draw=black, fill=myblue, drop shadow=myshadow]
\tikzset{every loop/.style={min distance=10mm,in=323,out=217,looseness=6}}
\pgfdeclarelayer{bg}
\pgfsetlayers{bg,main}

\let\epsilon\varepsilon

\author[1]{Mihály A. Vághy}
\author[1,2]{G\'abor Szederk\'enyi} 
\affil[1]{\small P\'azm\'any P\'eter Catholic University, Faculty of Information Technology and Bionics, Práter u. 50/a,  H-1083 Budapest, Hungary} 
\affil[2]{\small Systems and Control Laboratory, Institute for Computer Science and Control (SZTAKI), Kende u. 13-17, H-1111 Budapest, Hungary}
\title{Asymptotic stability of delayed complex balanced reaction networks with non-mass action kinetics}
\date{}

\begin{document}
\maketitle

\begin{abstract}
	We consider delayed chemical reaction networks with generalized kinetics of product form and show that complex balancing implies that within  each positive stoichiometric compatibility class there is a unique positive equilibrium that is locally asymptotically stable relative to its class. 
	The main tools of the proofs are respectively a version of the well-known classical logarithmic Lyapunov function applied to kinetic systems and its generalization to the delayed case as a Lyapunov-Krasovskii functional. Finally, we demonstrate our results through illustrative examples.
\end{abstract}


\section{Introduction}
Chemical reaction networks (also called CRNs or kinetic systems) are dynamical systems which can be formally represented as a set of (mathematically generalized) chemical reactions assuming certain reaction rates determining the velocity of the transformations of complexes to each other \cite{Feinberg2019,Chellaboina2009}. The scope of reaction networks reaches far beyond the (bio)chemical application field, since they can be considered as general descriptors of nonlinear dynamics capable of producing complex dynamical phenomena such as multiple equilibria, nonlinear oscillations, limit cycles, and even chaos \cite{Erdi1989}. It is known that majority of compartmental models used e.g., in population dynamics or epidemiology are naturally in kinetic form. Additionally, many other non-chemically motivated models can be algorithmically transformed to reaction network form \cite{Samardzija1989,Craciun2019}. Therefore, the results of chemical reaction network theory (CRNT) relating network structure and qualitative dynamics can be of general importance in the field of dynamical systems \cite{Angeli2009}.

Stability is a key qualitative property of dynamical models and their equilibria. In \cite{Horn1972}, the local stability of complex balanced equilibria of kinetic systems was shown using an entropy-like logarithmic Lyapunov function. The most well-known stability-related result in CRNT is probably the Deficiency Zero Theorem which states that weakly reversible deficiency zero CRNs are complex balanced independently of the (positive) values of reaction rate coefficients \cite{Feinberg1987}. According to the Global Attractor Conjecture, the stability of complex balanced networks is actually global within the nonnegative orthant \cite{craciun2015toric,Anderson2011}. The stability of a wide class of CRNs with more general kinetics than mass action was shown in \cite{Sontag2001}. These results were further extended in \cite{Chaves2005} for time-varying reaction rates using the notion of input-to-state stability.

The explicit modeling of time delays is often necessary to understand complex dynamical phenomena in nature or technology, and to build models having sufficient level of reliability \cite{stepan1989retarded}. Various phenomena may justify the inclusion of time delays into dynamical models such as protein expression time in systems biology \cite{zhu2024modeling}, hatching or maturation time in population dynamics \cite{ruiz2019attraction}, driver reaction times in traffic flow models \cite{orosz2009exciting}, latent periods in epidemic modeling \cite{wang2023dynamic}, or communication and feedback delays in complex networks \cite{zhu2020dynamical}. The most commonly used approach in the stability analysis of time-delay system is the construction of appropriate Lyapunov-Krasovskii functionals which is generally a challenging problem \cite{fridman2014introduction}.

The main motivation for introducing delayed chemical reactions was to focus on the most important species and chemical transformations, and to avoid the detailed description of mechanisms of less interest \cite{Roussel1996}. In delayed reactions, the consumption of reactant species is immediate, while the formation of products is delayed either through discrete or distributed delays. The notion of stoichiometric compatibility classes was generalized for delayed CRNs in \cite{Liptak2018}, and it was proved using a logarithmic Lyapunov-Krasovskii functional that complex balanced networks are at least locally stable for arbitrary finite delays. An analogous result for kinetic systems with distributed delays was given in \cite{liptak2019modelling}.
In \cite{Komatsu2019} the authors prove the generalization of well-known persistence results of chemical reaction networks \cite{Angeli2007,Angeli2011} to the delayed case and illustrate the applicability of the results on a biochemical reaction network. In \cite{Komatsu2020} the authors prove a delayed version of the deficiency zero theorem and discuss global asymptotic stability using the persistence results of \cite{Komatsu2019}. In \cite{Zhang2021} the authors consider delayed complex balanced chemical reaction networks with mass action kinetics and provide several sufficient conditions for the persistence of such systems using \cite{Komatsu2019}. The applicability of these results are further improved via semilocking set decomposition in \cite{Zhang2023}. 

Using the achievements outlined above, the purpose of the present paper is to further extend stability results for delayed complex balanced kinetic systems with general (non-mass action) kinetics. For this, an appropriate Lyapunov-Krasovskii functional will be proposed through which the local asymptotic stability and semistability of the positive equilibria can be shown.

The structure of the paper is as follows. Section \ref{sec:prelim} introduces the basic notions related to kinetic systems. In Section \ref{sec:QTD_CB}, the set of positive equilibria is studied in the context of complex balancing and the quasi-thermodynamic/thermostatic properties. The main contribution can be found in \ref{sec:delay}, where the semistability of positive complex balanced equilibria is shown. Section \ref{sec:examples} contains three computational examples to illustrate the theory. Finally, conclusions of the paper are given in Section \ref{sec:concl}.

\section{Preliminaries}\label{sec:prelim}
Throughout the paper $\mathbb{R}^N$, $\mathbb{R}_+^N$ and $\overline{\mathbb{R}}_+^N$ denotes the $N$-dimensional space of real, positive and nonnegative column vectors, respectively. For $x,y,\in\mathbb{R}_+^N$ the vector exponential $x^y$ is defined as $x^y=\prod_{k=1}^Nx_k^{y_k}$ and the inner product $x\cdot y$ is defined as $x\cdot y=\sum_{k=1}^Nx_iy_i$. For $x\in\mathbb{R}_+^N$ the vector logarithm $\log(x)$ is defined element-wise. For every $\tau\ge0$ we denote the Banach space of continuous functions mapping the interval $[-\tau,0]$ into $\mathbb{R}^N$, into $\mathbb{R}_+^N$ and into $\overline{\mathbb{R}}_+^N$ by $\mathcal{C}=C\qty\big([-\tau,0],\mathbb{R}^N)$, $\mathcal{C}_+$ and $\overline{\mathcal{C}}_+$, respectively. We equip the spaces $\mathcal{C}$, $\mathcal{C}_+$ and $\overline{\mathcal{C}}_+$ with the standard norm $\norm{\psi}=\sup_{s\in[-\tau,0]}\qty\big|\psi(s)|$, and the open ball around $\psi$ with radius $\epsilon>0$ is denoted by $\mathcal{B}_{\epsilon}(\psi)$.

We consider kinetic systems of the form
\begin{equation}\label{eq:ODE}
	\dot x(t)=\sum_{k=1}^M\kappa_k\gamma^{y_k}\qty\big(x(t))\qty\big(y_{k'}-y_k),
\end{equation}
where $x(t)\in\overline{\mathbb{R}}_+^N$ is the state vector, the function $\gamma:\overline{\mathbb{R}}_+^N\mapsto\overline{\mathbb{R}}_+^N$ is defined element-wise by the increasing functions $\gamma_i\in\mathcal{C}^1(\mathbb{R})$ that further satisfy $\gamma_i(0)=0$ and $\int_0^1|\log\gamma_i(s)|\dd{s}<\infty$. While in some cases, such as in the mass action case when $\gamma$ is the identity map, the functions $\gamma_i:\overline{\mathbb{R}}_+\mapsto\overline{\mathbb{R}}_+$ are onto, this property need not hold \cite[Section IV.B]{Sontag2001}. Let us assume that the maps $\gamma_i:\overline{\mathbb{R}}_+\mapsto[0,\sigma_i)$ are onto, where $0<\sigma_i\le\infty$ and that the limit
\begin{equation}\label{eq:gamma:prop}
	\lim_{x\uparrow\log\sigma_i}\int_a^x\gamma_i^{-1}(e^s)\dd{s}-bx=\infty
\end{equation}
holds for any finite $a<\log\sigma_i$ and any $b$. This general assumption also allows functions like $\gamma_i(s)=\frac{s}{c+s}$ for any $c>0$.

The system consists of $M$ reactions, where each reaction $k$ has a source and product complex with corresponding stoichiometric vectors $y_k,~y_{k'}\in\mathbb{N}^N$, respectively, and a positive reaction rate constant $\kappa_k$. The set of stoichiometric vectors is denoted with $\mathcal{K}$. In some cases we will use the complex matrix $Y$ that has the stoichiometric vectors as columns. The reaction vector of reaction $k$ is defined as $y_{k'}-y_k$. The linear span of the reaction vectors is called the stoichiometric subspace $\mathcal{S}$ of \eqref{eq:ODE}, defined as
\begin{equation}
	\mathcal{S}=\mathrm{span}\qty\big{y_{k'}-y_k\big|k=1,2,\dots,M}
\end{equation}
and for $p\in\mathbb{R}_+^N$ the corresponding positive stoichiometric compatibility class $\mathcal{S}_p$ is defined by
\begin{equation}
	\mathcal{S}_p=\qty\big{x\in\mathbb{R}_+^N\big|x-p\in\mathcal{S}}.
\end{equation}
It is well-known that the positive stoichiometric compatibility classes are positively invariant under \eqref{eq:ODE}; that is, we have that $x(t)\in\mathcal{S}_p$ for $t\ge$ if $x(0)\in\mathcal{S}_p$.

We note that \eqref{eq:ODE} can be rewritten in matrix form as follows. Assume that the number of distinct complexes is $L$ and define $\kappa_{ij}$ as $\kappa_k$ if there is a reaction $k$ such that $y_{k'}=y_j$ and $y_k=y_i$, and zero otherwise. Denoting by $K$ the matrix defined element-wise as $[K]_{ij}=\kappa_{ij}$, the system \eqref{eq:ODE} takes the form
\begin{equation}\label{eq:ODE_matrix}
	\dot x(t)=Y\qty\big(K-\mathrm{diag}(1_L^{\top}K))\Gamma(x)=:Y\tilde K\Gamma(x),
\end{equation}
where $1_L\in\mathbb{R}^L$ denotes a matrix with all of its coordinates equal to one and $\Gamma:\overline{\mathbb{R}}_+^N\mapsto\overline{\mathbb{R}}_+^N$ is defined as
\begin{equation}
	\Gamma(x)=\qty\big[\gamma^{y_1}(x)~\gamma^{y_2}(x)~\cdots~\gamma^{y_L}(x)].
\end{equation}
Note, that $\tilde K$ is the weighted negative Laplacian of the reaction graph of the system.

\section{Quasi-thermodynamic property and complex balancing}\label{sec:QTD_CB}
In this section, we restate some of the stability results described in \cite{Sontag2001} under milder conditions using the computation approach of \cite{Feinberg2019}. 
Here we consider nondelayed kinetic systems of the form \eqref{eq:ODE}. First, let us recall some definitions. A positive vector $\overline x\in\mathbb{R}_+^N$ is called a positive equilibrium of \eqref{eq:ODE} if $x(t)\equiv\overline x$ is a solution of \eqref{eq:ODE}; that is, the equilibria of \eqref{eq:ODE} satisfy the equation
\begin{equation}
	f(\overline x):=\sum_{k=1}^M\kappa_k\gamma^{y_k}(\overline x)\qty\big(y_{k'}-y_k)=0,
\end{equation}
where $f:\overline{\mathbb{R}}_+^N\mapsto\mathcal{S}$ denotes the species formation rate function of the kinetic system \eqref{eq:ODE}.
In the classical terminology of \cite{Horn1972,Feinberg2019} a kinetic system is called quasi-thermostatic if there exists a positive vector $\overline x\in\mathbb{R}_+^N$ such that the set of positive equilibria is identical to the set
\begin{equation}
	\mathcal{E}=\qty\big{\tilde x\in\mathbb{R}_+^N\big|\log(\tilde x)-\log(\overline x)\in\mathcal{S}^{\perp}}.
\end{equation}
In this case we say that the kinetic system is \textit{quasi-thermostatic} with respect to $\overline x$. Standard arguments show that then the system is quasi-thermostatic with respect to any element of $\mathcal{E}$. The distribution of positive equilibria of quasi-thermostatic systems can be efficiently characterized, namely, each positive stoichiometric compatibility class contains precisely one positive equilibrium \cite{Horn1972}.

Furthermore, a kinetic system is called \textit{quasi-thermodynamic} if there exists an $\overline x\in\mathbb{R}_+^N$ such that the system is quasi-thermostatic with respect to $\overline x$, and
\begin{equation}\label{eq:mas:Helmholtz}
	\qty\big(\log(x)-\log(\overline x))\cdot f(x)\le0
\end{equation}
holds for $x\in\mathbb{R}_+^N$, with equality holding only if $f(x)=0$ or, equivalently, if $\rho(x)-\rho(\overline x)\in\mathcal{S}^{\perp}$. In this case we say that the kinetic system is quasi-thermodynamic with respect to $\overline x$. Similarly to quasi-thermostaticity, a system is quasi-thermodynamic with respect to any element of $\mathcal{E}$. The main consequence of quasi-thermodynamicity is that the unique positive equilibrium of each positive stoichiometric compatibility class is asymptotically stable relative to its class. This arises from the fact that the gradient of the function
\begin{equation}
	H(x,\overline x)=\sum_{i=1}^Nx_i(\log x_i-\log\overline x_i-1)
\end{equation}
is given by $\log(x)-\log(\overline x)$ which is a term in Eq. \eqref{eq:mas:Helmholtz}. Thus, the function
\begin{equation}\label{eq:mas:Lyapunov}
	V(x,\overline x)=\sum_{i=1}^N\qty\big(x_i(\log x_i-\log\overline x_i-1)+\overline x_i)=\sum_{i=1}^N\qty\bigg(x_i\log\frac{x_i}{\overline x_i}+\overline x_i-x_i)
\end{equation}
is a Lyapunov function for quasi-thermodynamic kinetic models. The short physical background of this is that $H$ was used to describe the Helmholtz free energy density of the system, and its gradient is the chemical potential function.

As noted in \cite{Horn1972}, while the above definition is physically associated with mass action kinetics and ideal gas mixtures, it could apply to any kinetic system. In some cases the definitions can be extended without voiding their consequences. In order to do so, following \cite{Sontag2001}, we define for $x\in\mathbb{R}_+^N$ the function
\begin{equation}
	\rho(x)=\log\qty\big(\gamma(x)),
\end{equation}
where $\gamma$ is defined as in Eq. \eqref{eq:ODE}.
A kinetic system of the form \eqref{eq:ODE} is called \textit{quasi-thermostatic in the generalized sense} if there exists an $\overline x\in\mathbb{R}_+^N$ such that the set of positive equilibria is identical to the set
\begin{equation}\label{eq:E_char}
	\mathcal{E}=\qty\big{\tilde x\in\mathbb{R}_+^N\big|\rho(\tilde x)-\rho(\overline x)\in\mathcal{S}^{\perp}}.
\end{equation}
For brevity, we simply say that the kinetic system is quasi-thermostatic with respect to $\overline x$. Again, similarly to classical quasi-thermostaticity, standard arguments show that then the system is quasi-thermostatic with respect to any element of $\mathcal{E}$. Furthermore, the distribution of the positive equilibria of quasi-thermostatic kinetic systems across positive stoichiometric compatibility classes can be characterized. We describe that distribution in the following proposition.

\begin{prop}\label{prop:ode:one}
	Assume that the kinetic system \eqref{eq:ODE} is quasi-thermostatic. Then, for every $p\in\mathbb{R}_+^N$ the corresponding positive stoichiometric compatibility class $\mathcal{S}_p$ contains precisely one positive equilibrium.
\end{prop}
\begin{proof}
	We first show the existence of a point in $\mathcal{S}_p\cap\mathcal{E}$. Let $\overline x$ be an element of $\mathcal{E}$. By \cite[Proposition B.1]{Feinberg1995} there exists a (unique) vector $\mu\in\mathcal{S}^{\perp}$ such that
	\begin{equation}
		\gamma(\overline x)e^{\mu}-p\in\mathcal{S}.
	\end{equation}
	Let $\tilde x$ be defined by
	\begin{equation}
		\gamma(\tilde x):=\gamma\qty(\overline x)e^{\mu}.
	\end{equation}
	Then $\tilde x\in\mathcal{S}_p$ and taking logarithm shows that
	\begin{equation}
		\rho(\tilde x)-\rho(\overline x)=\mu\in\mathcal{S}^{\perp};
	\end{equation}
	that is, we have that $\tilde x\in\mathcal{E}$ as well.

	In order to show uniqueness, let us assume by contradiction that $\tilde x$ and $\overline x$ are distinct positive equilibria in $\mathcal{S}_p$. Then $\tilde x-\overline x\in\mathcal{S}$ and $\rho(\tilde x)-\rho(\overline x)\in\mathcal{S}^{\perp}$, and thus
	\begin{equation}
		0=\qty\big(\rho(\tilde x)-\rho(\overline x))\cdot(\tilde x-\overline x)=\sum_{i=1}^N\qty\big(\log\gamma_i(\tilde x_i)-\log\gamma_i(\overline x_i))(\tilde x_i-\overline x_i).
	\end{equation}
	Since the functions $\gamma_i$ and the logarithm are strictly increasing, the above expression is zero if and only if $\tilde x=\overline x$.
\end{proof}

\begin{rem}
	Note that we implicitly used the assumption \eqref{eq:gamma:prop}, see \cite[Lemma IV.1]{Sontag2001} and Proposition \ref{prop:dde:one} for more details.
\end{rem}

A kinetic system of the form \eqref{eq:ODE} is called \textit{quasi-thermodynamic in the generalized sense} if there exists an $\overline x\in\mathbb{R}_+^N$ such that the system is quasi-thermostatic with respect to $\overline x$ and
\begin{equation}\label{eq:gen:Helmholtz}
	\qty\big(\rho(x)-\rho(\overline x))\cdot f(x)\le0
\end{equation}
holds for $x\in\mathbb{R}_+^N$, where equality holds only if $f(x)=0$ or, equivalently, if $\rho(x)-\rho(\overline x)\in\mathcal{S}^{\perp}$. Again, for brevity, we simply say that the kinetic system is quasi-thermodynamic with respect to $\overline x$, however, similarly to quasi-thermostaticity, a system is quasi-thermodynamic with respect to any element of $\mathcal{E}$.

The following proposition and its proof shows that the underlying function
\begin{equation}\label{eq:gen:Lyapunov}
	V(x,\overline x)=\sum_{i=1}^N\int_{\overline x_i}^{x_i}\qty\big(\log\gamma_i(s)-\log\gamma_i(\overline x_i))\dd{s}
\end{equation}
is a Lyapunov function of the system \eqref{eq:ODE}. Note, that \eqref{eq:gen:Lyapunov} reduces to \eqref{eq:mas:Lyapunov} in the mass action case.

\begin{prop}
	Assume that the kinetic system \eqref{eq:ODE} is quasi-thermodynamic. Then, each positive stoichiometric compatibility class contains precisely one positive equilibrium and that equilibrium is asymptotically stable, and there is no nontrivial periodic trajectory along which all species concentrations are positive.
\end{prop}
\begin{proof}
	The fact that each positive stoichiometric compatibility class contains precisely one positive equilibrium follows from quasi-thermostaticity.

	Let us consider any positive stoichiometric compatibility class $\mathcal{S}_p$ and denote its unique positive equilibrium by $\overline x$. Then, for any $x\in\mathcal{S}_p$ other than $\overline x$, we have that
	\begin{equation}\label{eq:QTD:strict}
		\qty\big(\rho(x)-\rho(\overline x))\cdot f(x)<0.
	\end{equation}
	It is easy to see that $V(x,\overline x)\ge0$ and equality holds only if $x=\overline x$, and that $\nabla V(x,\overline x)=\rho(x)-\rho(\overline x)$. This, combined with \eqref{eq:QTD:strict} show that 
	\begin{equation}
		\nabla V(x,\overline x)\cdot f(x)<0
	\end{equation}
	holds for any $x\in\mathcal{S}_p$ other than $\overline x$. Standard arguments show that $V(x,\overline x)$ is a strict Lyapunov function for $\overline x$ on its positive stoichiometric compatibility class $\mathcal{S}_p$, thus $\overline x$ is asymptotically stable relative to $\mathcal{S}_p$.

	To show that no nontrivial periodic trajectories can exist along which all species concentrations are positive, assume by contradiction that $x:[0,T]\mapsto\mathbb{R}_+^N$ is such a solution with $x(T)=x(0)$ and denote the unique positive equilibrium of the corresponding positive stoichiometric compatibility class by $\overline x$. Then
	\begin{equation}
		V\qty\big(x(T),\overline x)-V\qty\big(x(0),\overline x)=\int_0^T\nabla V\qty\big(x(t),\overline x)\cdot f\qty\big(x(t))\dd{t}<0,
	\end{equation}
	and thus
	\begin{equation}
		V\qty\big(x(T),\overline x)<V\qty\big(x(0),\overline x),
	\end{equation}
	contradicting $x(T)=x(0)$.
\end{proof}

In \cite{Sontag2001} the author considers systems of the form \eqref{eq:ODE} or, equivalently, of the form \eqref{eq:ODE_matrix}, and assumes that the complex matrix $Y$ is of full rank and none of its rows vanishes, and that $\tilde K$ is irreducible (implying that the reaction graph is strongly connected). Then, without using the above terminology, the author shows that such systems are quasi-thermodynamic. We note, that these assumptions imply that if $\overline x$ is an equilibrium of \eqref{eq:ODE_matrix}, then $\tilde K\Gamma(\overline x)=0$; that is, the vector $\Gamma(\overline x)$ is in the kernel of $\tilde K$. Thus, systems that satisfy the above assumptions are complex balanced, defined as follows.

Without any restrictions on $Y$ or assuming that $\tilde K$ is irreducible, an equilibrium $\overline x$ is called \textit{complex balanced} if $\tilde K\Gamma(\overline x)=0$ or, equivalently, if for every complex $\eta\in\mathcal{K}$ we have that
\begin{equation}
	\sum_{k:\eta=y_k}\kappa_k\gamma^{y_k}(\overline x)=\sum_{k:\eta=y_{k'}}\kappa_k\gamma^{y_k}(\overline x),
\end{equation}
where the sum on the left-hand side is taken over the reactions where $\eta$ is the source complex and the sum on the right-hand side is taken over the reactions where $\eta$ is the product complex. Therefore, complex balanced equilibria are also called vertex-balanced in the literature \cite{muller2023new}. We note that this setting is indeed more general than that of \cite{Sontag2001}, as for mass action systems complex balancing can occur in weakly reversible systems, not just in strongly connected systems; that is, there can be more than one linkage classes. 

First, we show that the existence of a positive complex balanced equilibrium affects every positive equilibrium.

\begin{prop}
	Assume that the kinetic system \eqref{eq:ODE} admits a positive complex balanced equilibrium. Then every positive equilibrium is complex balanced.
\end{prop}
\begin{proof}
	Let us assume that $\overline x\in\mathbb{R}_+^N$ is a positive complex balanced equilibrium and $\tilde x\in\mathbb{R}_+^N$ is a positive equilibrium other than $\overline x$. Then $\tilde x\in\mathcal{E}$; that is, we have that $\rho(\tilde x)-\rho(\overline x)\in\mathcal{S}^{\perp}$. Let us define for $k=1,2,\dots,M$ the function $q_k:\mathbb{R}_+^N\mapsto\mathbb{R}$ by
	\begin{equation}
		q_k(x)=\qty\big(\rho(x)-\rho(\overline x))\cdot y_k.
	\end{equation}
	Then, for any complex $\eta\in\mathcal{K}$ we have that
	\begin{equation}
		\begin{aligned}
			\sum_{k:\eta=y_k}\kappa_k\gamma^{y_k}(\tilde x)-\sum_{k:\eta=y_{k'}}\kappa_k\gamma^{y_k}(\tilde x)&=\sum_{k:\eta=y_k}\kappa_k\gamma^{y_k}(\overline x)e^{q_k(\tilde x)}-\sum_{k:\eta=y_{k'}}\kappa_k\gamma^{y_k}(\overline x)e^{q_k(\tilde x)}\\
			&=e^{q_{\eta}(\tilde x)}\qty\Bigg(\sum_{k:\eta=y_k}\kappa_k\gamma^{y_k}(\overline x)-\sum_{k:\eta=y_{k'}}\kappa_k\gamma^{y_k}(\overline x))=0,
		\end{aligned}
	\end{equation}
	thus $\tilde x$ is indeed complex balanced.
\end{proof}

Finally, the connection between complex balanced systems and quasi-thermodynamic systems are described in the following proposition.

\begin{prop}
	Assume that the kinetic system \eqref{eq:ODE} is complex balanced. Then it is quasi-thermodynamic.
\end{prop}
\begin{proof}
	Let us consider the positive complex balanced equilibrium $\overline x$; that is, the equality
	\begin{equation}
		\sum_{k:\eta=y_k}\kappa_k\gamma^{y_k}(\overline x)=\sum_{k:\eta=y_{k'}}\kappa_k\gamma^{y_k}(\overline x)
	\end{equation}
	holds for any complex $\eta\in\mathcal{K}$. Observe that for any $x\in\mathbb{R}_+^N$ we have that
	\begin{equation}
		\qty\big(\rho(x)-\rho(\overline x))\cdot f(x)=\sum_{k=1}^M\kappa_k\gamma^{y_k}(x)\qty\big(q_{k'}(x)-q_k(x))=\sum_{k=1}^M\kappa_k\gamma^{y_k}(\overline x)e^{q_k(x)}\qty\big(q_{k'}(x)-q_k(x)).
	\end{equation}
	Using the well-known inequality
	\begin{equation}\label{eq:exp_ab}
		e^a(b-a)\le e^b-e^a
	\end{equation}
	leads to
	\begin{equation}\label{eq:CB->QTD:Lyapunov}
		\begin{aligned}
			\qty\big(\rho(x)-\rho(\overline x))\cdot f(x)&\le\sum_{k=1}^M\kappa_k\gamma^{y_k}(\overline x)\qty\big(e^{q_{k'}(x)}-e^{q_k(x)})\\
			&=\sum_{\eta\in\mathcal{K}}e^{q_{\eta}(x)}\qty\Bigg(\sum_{k:\eta=y_{k'}}\kappa_k\gamma^{y_k}(\overline x)-\sum_{k:\eta=y_k}\kappa_k\gamma^{y_k}(\overline x))=0,
		\end{aligned}
	\end{equation}
	where equality holds if and only if $q_{k'}(x)=q_k(x)$ for each reaction $k=1,2,\dots,M$; that is, if and only if $\rho(x)-\rho(\overline x)$ lies in $\mathcal{S}^{\perp}$. In particular, if $f(x)=0$, then $\rho(x)-\rho(\overline x)$ lies in $\mathcal{S}^{\perp}$. It remains to be shown that if $\rho(x)-\rho(\overline x)$ lies in $\mathcal{S}^{\perp}$, then $f(x)=0$, as a quasi-thermodynamic system needs to be quasi-thermostatic as well. Rewrite the species formation rate function as
	\begin{equation}
		\begin{aligned}
			f(x)&=\sum_{\eta\in\mathcal{K}}\eta\qty\Bigg(\sum_{k:\eta=y_{k'}}\kappa_k\gamma^{y_k}(x)-\sum_{k:\eta=y_k}\kappa_k\gamma^{y_k}(x))\\
			&=\sum_{\eta\in\mathcal{K}}\eta\qty\Bigg(\sum_{k:\eta=y_{k'}}\kappa_k\gamma^{y_k}(\overline x)e^{q_k(x)}-\sum_{k:\eta=y_k}\kappa_k\gamma^{y_k}(\overline x)e^{q_k(x)}).
		\end{aligned}
	\end{equation}
	If $x$ is such that $\rho(x)-\rho(\overline x)\in\mathcal{S}^{\perp}$, then $\rho(x)-\rho(\overline x)$ is orthogonal to every reaction vector, and thus
	\begin{equation}
		f(x)=\sum_{\eta\in\mathcal{K}}e^{q_{\eta}(x)}\eta\qty\Bigg(\sum_{k:\eta=y_{k'}}\kappa_k\gamma^{y_k}(\overline x)-\sum_{k:\eta=y_k}\kappa_k\gamma^{y_k}(\overline x))=0;
	\end{equation}
	that is, the vector $x$ is an equilibrium. This shows that the set of positive equilibria coincides with the set $\mathcal{E}$, and thus the system is quasi-thermostatic. This, combined with \eqref{eq:CB->QTD:Lyapunov} shows that the system is quasi-thermodynamic as well.
\end{proof}

\section{Stability of delayed kinetic models}\label{sec:delay}
In this section, we consider kinetic systems with delayed reactions having the form
\begin{equation}\label{eq:DDE}
	\dot x(t)=\sum_{k=1}^M\kappa_k\qty\Big(\gamma^{y_k}\qty\big(x(t-\tau_k))y_{k'}-\gamma^{y_k}\qty\big(x(t))y_k),
\end{equation}
where $\tau_k\ge0$ are discrete constant time delays. The solution corresponding to an initial function $\psi\in\overline{\mathcal{C}}_+$ at time $t\ge0$ is denoted by $x^{\psi}(t)\in\overline{\mathbb{R}}_+^N$ or by $x_t^{\psi}\in\overline{\mathcal{C}}_+$ when we use it as a function.

First, we extend the notion of positive stoichiometric compatibility classes to the delayed kinetic system \eqref{eq:DDE}. We note, that the following definition and invariance proof was already established in \cite{Liptak2018} in the case of mass action kinetics and in \cite{Komatsu2020} in the general case. For each $v\in\mathbb{R}^{N}$ define the functional $c_v:\mathcal{C}_+\mapsto\mathbb{R}$ as
\begin{equation}
	c_v(\psi)=v\cdot\qty\Bigg[\psi(0)+\sum_{k=1}^M\qty\Bigg(\kappa_k\int_{-\tau_k}^0\gamma^{y_k}\qty\big(\psi(s))\dd{s})y_k],\qquad\psi\in\mathcal{C}_+.
\end{equation}
For each $\theta\in\mathcal{C}_+$ the positive stoichiometric compatibility class of \eqref{eq:DDE} corresponding to $\theta$ is denoted by $\mathcal{D}_{\theta}$ and is defined by
\begin{equation}
	\mathcal{D}_{\theta}=\qty\big{\psi\in\mathcal{C}_+\big|c_v(\psi)=c_v(\theta)\text{ for all }v\in\mathcal{S}^{\perp}}.
\end{equation}
Clearly, $\psi\in\mathcal{D}_{\theta}$ if and only if $\psi\in\mathcal{C}_+$ and
\begin{equation}\label{eq:Dt_char}
	\psi(0)-\theta(0)+\sum_{k=1}^M\qty\Bigg(\kappa_k\int_{-\tau_k}^0\qty\Big(\gamma^{y_k}\qty\big(\psi(r))-\gamma^{y_k}\qty\big(\theta(s)))\dd{s})y_k\in\mathcal{S}.
\end{equation}
This shows that if each delay $\tau_k$ is zero, then the delayed positive stoichiometric compatibility classes reduce to the positive compatibility classes of \eqref{eq:ODE}.

The following Proposition established the invariance property of $\mathcal{D}_{\theta}$.

\begin{prop}
	For every $\theta\in\mathcal{C}_+$ the positive stoichiometric compatibility class $\mathcal{D}_{\theta}$ is a closed subset of $\mathcal{C}_+$. Moreover, $\mathcal{D}_{\theta}$ is positively invariant under \eqref{eq:DDE}; that is, if $\psi\in\mathcal{D}_{\theta}$, then $x_t^{\psi}\in\mathcal{D}_{\theta}$ for all $t\ge0$.
\end{prop}
\begin{proof}
	The closedness follows from the continuity of $c_v$. We will show that for each $v\in\mathcal{S}^{\perp}$ the functional $c_v$ is constant along the trajectories of \eqref{eq:DDE}. To see this, let us assume that $x$ is a solution of \eqref{eq:DDE}. Then for $t\ge0$ we have that
	\begin{equation}
		\begin{aligned}
			\dv{}{t}c_v(x_t)&=v\cdot\qty\Bigg(\dot x(t)+\sum_{k=1}^M\kappa_k\qty\Big(\gamma^{y_k}\qty\big(x(t))-\gamma^{y_k}\qty\big(x(t-\tau_k)))y_k)\\
			&=v\cdot\qty\Bigg(\sum_{k=1}^M\kappa_k\gamma^{y_k}\qty\big(x(t-\tau_k))(y_{k'}-y_k))=\sum_{k=1}^M\kappa_k\gamma^{y_k}\qty\big(x(t-\tau_k))v\cdot(y_{k'}-y_k)=0,
		\end{aligned}
	\end{equation}
	where the last equality follows from the fact that $v\in\mathcal{S}^{\perp}$. Thus, if $\psi\in\mathcal{D}_{\theta}$, then for every $v\in\mathcal{S}^{\perp}$ and $t\ge0$ the equalities
	\begin{equation}
		c_v\qty\big(x_t^{\psi})=c_v\qty\big(x_0^{\psi})=c_v(\psi)=c_v(\theta)
	\end{equation}
	hold, showing that $x_t^{\psi}\in\mathcal{D}_{\theta}$ as desired.
\end{proof}

A positive vector $\overline x\in\mathbb{R}_+^N$ is called a positive equilibrium of \eqref{eq:DDE} if $x(t)\equiv\overline x$ is a solution of \eqref{eq:DDE}; that is, the equilibria of \eqref{eq:DDE} and \eqref{eq:ODE} coincide. The following proposition is the generalization of Proposition \ref{prop:ode:one} for complex balanced systems.

\begin{prop}\label{prop:dde:one}
	Assume that the kinetic system \eqref{eq:DDE} is complex balanced. Then, for every $\theta\in\mathcal{C}_+$ the corresponding delayed positive stoichiometric compatibility class $\mathcal{D}_{\theta}$ of the system \eqref{eq:DDE} contains precisely one positive equilibrium.
\end{prop}
\begin{proof}
	While in the nondelayed case (see Proposition \ref{prop:ode:one}) existence is shown via \cite[Proposition B.1]{Feinberg1995} without modification, in the delayed case we need to adapt certain steps of the proof based on \cite[Theorem 4.4]{Komatsu2020}.

	Let us for $\overline x\in\mathcal{E}$ define the positive vector $b\in\mathbb{R}_+^N$ by
	\begin{equation}
		b_i=\theta_i(0)+\sum_{k=1}^M\kappa_k\int_{-\tau_k}^0\gamma^{y_k}\qty\big(\theta(s))\dd{s}
	\end{equation}
	and the continuously differentiable function $g:\mathbb{R}^N\mapsto\mathbb{R}$ by
	\begin{equation}
		g(x)=\sum_{i=1}^N\qty\Bigg(\int_0^{x_i}\gamma_i^{-1}\qty\big(\gamma_i(\overline x_i)e^{s})\dd{s}+\overline x_i-b_ix_i)+\sum_{k=1}^M\kappa_k\tau_k\qty\big(\gamma(\overline x)e^x)^{y_k}.
	\end{equation}
	We note that adding $\overline x_i$ to the integral is not necessary for the following analysis, but then $g(x)$ reduces precisely to the analogous function in the known proof of this theorem for mass action systems.

	The gradient of $g$ is given by
	\begin{equation}
		\nabla g(x)=\gamma^{-1}\qty\Big(\gamma(\overline x)e^x)-b+\sum_{k=1}^M\kappa_k\tau_k\qty\big(\gamma(\overline x)e^x)^{y_k}y_k		
	\end{equation}
	and that the Hessian of $g$ is written as
	\begin{equation}
		H_g(x)=\mathrm{diag}\qty\Bigg(\frac{\gamma(\overline x)e^x}{\gamma'\qty\Big(\gamma^{-1}\qty\big(\gamma(\overline x)e^x))})+\sum_{k=1}^M\kappa_k\tau_k\qty\big(\gamma(\overline x)e^x)^{y_k}y_ky_k^{\top},
	\end{equation}
	where the fraction in the diagonal matrix is defined element-wise. The corresponding quadratic form is positive-definite as the first term is a diagonal matrix with positive entries, and thus is positive-definite, and the second term consists of positive factors and the positive-semidefinite matrix $y_ky_k^{\top}$. Then the function $g$ is strictly convex everywhere.

	From the property \eqref{eq:gamma:prop} of the $\gamma_i$ functions it follows that for any nonzero vector $x\in\mathbb{R}^N$ we have that
	\begin{equation}
		\lim_{a\rightarrow\infty}\qty\Bigg(\int_0^{x_i}\gamma_i^{-1}\qty\big(\gamma_i(\overline x_i)e^{as})\dd{s}+\overline x_i-ab_ix_i)=\begin{cases}\infty,\qquad&x_i\neq0,\\\overline x_i\qquad&x_i=0,\end{cases}
	\end{equation}
	and thus
	\begin{equation}\label{eq:gax}
		\lim_{a\rightarrow\infty}\sum_{i=1}^N\qty\Bigg(\int_0^{x_i}\gamma_i^{-1}\qty\big(\gamma_i(\overline x_i)e^{as})\dd{s}+\overline x_i-ab_ix_i)\le\lim_{a\rightarrow\infty}g(ax)=\infty.
	\end{equation}

	Let $\overline g:\mathcal{S}^{\perp}\mapsto\mathbb{R}$ be the restriction of $g$ to $\mathcal{S}^{\perp}$, which is also continuously differentiable and strictly convex. Define the subset
	\begin{equation}
		\mathcal{S}^{\perp}\supset G=\qty\big{x\in\mathcal{S}^{\perp}\big|\overline g(x)\le g(0)}.
	\end{equation}
	Clearly $G$ is convex, closed in $\mathbb{R}^N$, contains the $0$ zero vector and contains no half line with endpoint $0$ because of \eqref{eq:gax}. Then $G$ is bounded, and thus compact as well, since in a finite-dimensional vector space every unbounded closed convex set containing $0$ must contain a half line with endpoint $0$ \cite[Theorem 3.5.1]{Stoer1970}. The continuity of $\overline g$ and the compactness of $G$ implies that there exists $\mu\in G$ such that
	\begin{equation}
		\overline g(\mu)\le\overline g(x),\qquad\forall x\in G.
	\end{equation}
	In fact, $\overline g(0)<\overline g(x)$ for $x\in\mathcal{S}^{\perp}\backslash G$, and thus
	\begin{equation}
		\overline g(\mu)\le\overline g(x),\qquad\forall x\in\mathcal{S}^{\perp}.
	\end{equation}
	Then for $\xi\in\mathcal{S}^{\perp}$, the equality
	\begin{equation}
		0=\dv{}{t}\overline g(\mu+t\xi)\bigg|_{t=0}=\dv{}{t}g(\mu+t\xi)\bigg|_{t=0}=\nabla g(\mu)\cdot\xi
	\end{equation}
	holds; that is, the vector $\nabla g(\mu)$ is in $\mathcal{S}$, and thus
	\begin{equation}
		\begin{aligned}
			&\gamma^{-1}\qty\big(\gamma(\overline x)e^{\mu})-b+\sum_{k=1}^M\kappa_k\tau_k\qty\big(\gamma(\overline x)e^{\mu})^{y_k}y_k\\
			&=\gamma^{-1}\qty\big(\gamma(\overline x)e^{\mu})-\theta(0)+\sum_{k=1}^M\qty\Bigg(\kappa_k\int_{-\tau_k}^0\qty\Big(\qty\big(\gamma(\overline x)e^{\mu})^{y_k}-\gamma^{y_k}\qty\big(\theta(s)))\dd{s})y_k\in\mathcal{S}
		\end{aligned}
	\end{equation}

	Let $\tilde x$ be defined by
	\begin{equation}
		\tilde x=\gamma^{-1}\qty\big(\gamma\qty(\overline x)e^{\mu}).
	\end{equation}
	Then $\tilde x\in\mathcal{D}_{\theta}$ and taking logarithm shows that
	\begin{equation}
		\rho(\tilde x)-\rho(\overline x)=\mu\in\mathcal{S}^{\perp};
	\end{equation}
	that is, we have that $\tilde x\in\mathcal{E}$ as well.

	To show uniqueness, assume by contradiction that $\tilde x$ and $\overline x$ are distinct positive equilibria in $\mathcal{D}_{\theta}$. Then by \eqref{eq:Dt_char} it follows that
	\begin{equation}
		\tilde x-\overline x+\sum_{k=1}^M\qty\Bigg(\kappa_k\int_{-\tau_k}^0\qty\big(\gamma^{y_k}(\tilde x)-\gamma^{y_k}(\overline x))\dd{s})y_k\in\mathcal{S}.
	\end{equation}
	This, combined with the characterization \eqref{eq:E_char} shows that
	\begin{equation}
		\begin{aligned}
			0&=\qty\big(\rho(\tilde x)-\rho(\overline x))\cdot\qty\Bigg[\tilde x-\overline x+\sum_{k=1}^M\qty\Bigg(\kappa_k\int_{-\tau_k}^0\qty\big(\gamma^{y_k}(\tilde x)-\gamma^{y_k}(\overline x))\dd{s})y_k]\\
			&=\sum_{i=1}^N\qty\big(\log\gamma(\tilde x_i)-\log\gamma(\overline x_i))(\tilde x_i-\overline x_i)\\
			&+\sum_{k=1}^M\qty\Bigg(\kappa_k\tau_k\qty\big(\log\gamma^{y_k}(\tilde x)-\log\gamma^{y_k}(\overline x))\qty\big(\gamma^{y_k}(\tilde x)-\gamma^{y_k}(\overline x))).
		\end{aligned}
	\end{equation}
	Since the functions $\gamma_i$ and the logarithm are strictly increasing, the above expression is zero if and only if $\tilde x=\overline x$.
\end{proof}

The following theorem and the underlying Lyapunov-Krasovskii functional is the main contribution of the paper.

\begin{thm}\label{thm:Lyapunov}
	Assume that the kinetic system \eqref{eq:DDE} is complex balanced. Then, every positive equilibrium of the system is locally asymptotically stable relative to its positive stoichiometric compatibility class.
\end{thm}
\begin{proof}
	Consider the candidate Lyapunov--Krasovskii functional $V:\mathcal{C}_+\mapsto\overline{\mathbb{R}}_+$ defined for $\psi\in\mathcal{C}_+$ by
	\begin{equation}\label{eq:V}
		\begin{aligned}
			V(\psi)&:=V(\psi,\overline x)=\sum_{i=1}^N\int_{\overline x_i}^{\psi_i(0)}\qty\big(\log\gamma_i(s)-\log\gamma_i(\overline x_i))\dd{s}\\
			&+\sum_{k=1}^M\kappa_k\int_{-\tau_k}^0\qty\bigg(\gamma^{y_k}\qty\big(\psi(s))\qty\Big(\log\gamma^{y_k}\qty\big(\psi(s))-\log\gamma^{y_k}(\overline x)-1)+\gamma^{y_k}(\overline x))\dd{s}.
		\end{aligned}
	\end{equation}
	Using \eqref{eq:exp_ab} shows that the second term of \eqref{eq:V} is nonnegative and zero if only if $x=\overline x$, while in \cite{Sontag2001} the author shows the same for the first term. The gradient of the first term of \eqref{eq:V} is $\rho(x)-\rho(\overline x)$, and thus the Lyapunov-Krasovskii directional derivative along trajectories of \eqref{eq:DDE} is given by
	\begin{equation}
		\begin{aligned}
			\dot V(x_t)&=\sum_{k=1}^M\kappa_k\qty\Big(\gamma^{y_k}\qty\big(x(t-\tau_k))q_{k'}(t)-\gamma^{y_k}\qty\big(x(t))q_k(t))\\
			&+\sum_{k=1}^M\kappa_k\qty\Big(\gamma^{y_k}\qty\big(x(t))q_k(t)-\gamma^{y_k}\qty\big(x(t-\tau_k))q_k(t-\tau_k))\\
			&+\sum_{k=1}^M\kappa_k\qty\Big(\gamma^{y_k}\qty\big(x(t-\tau_k))-\gamma^{y_k}\qty\big(x(t)))\\
			&=\sum_{k=1}^M\kappa_k\qty\Big(\gamma^{y_k}\qty\big(x(t-\tau_k))\qty\big(q_{k'}(t)-q_k(t-\tau_k))+\gamma^{y_k}\qty\big(x(t-\tau_k))-\gamma^{y_k}\qty\big(x(t))).
		\end{aligned}
	\end{equation}
	Rewrite the above as
	\begin{equation}
		\dot V(x_t)=\sum_{k=1}^M\kappa_k\gamma^{y_k}(\overline x)\qty\Big(e^{q_k(t-\tau_k)}\qty\big(q_{k'}(t)-q_k(t-\tau_k))+e^{q_k(t-\tau_k)}-e^{q_k(t)})
	\end{equation}
	and use inequality \eqref{eq:exp_ab} to find that
	\begin{equation}
		\begin{aligned}
			\dot V(x_t)&\le\sum_{k=1}^M\kappa_k\gamma^{y_k}(\overline x)\qty\big(e^{q_{k'}(t)}-e^{q_k(t)})\\
			&=\sum_{\eta\in\mathcal{K}}e^{q_{\eta}(t)}\qty\Bigg(\sum_{k:\eta=y_{k'}}\kappa_k\gamma^{y_k}(\overline x)-\sum_{k:\eta=y_k}\kappa_k\gamma^{y_k}(\overline x))=0,
		\end{aligned}
	\end{equation}
	as the system is complex balanced, and $\dot V(x_t)=0$ if and only if the equality
	\begin{equation}
		q_{k'}(t)=q_k(t-\tau_k)
	\end{equation}
	holds for each reaction $k=1,2,\dots,M$. Standard arguments (see \cite[Theorem 3]{Liptak2018}) show that the largest invariant subset of the set
	\begin{equation}
		\mathcal{R}=\qty\Big{\psi\in\mathcal{C}_+\Big|\dot V(\psi)=0}=\qty\Big{\psi\in\mathcal{C}_+\Big|q_{k'}(t)=q_k(t-\tau_k)\text{ for }k=1,2,\dots,M}
	\end{equation}
	consists of constant functions that are positive complex balanced equilibria, and the proof is finished.
\end{proof}
We note that similarly to the nondelayed case, the quasi-thermostatic and quasi-thermodynamic properties could be defined in order to generalize these stability notions to not just complex balanced systems. However, the physical interpretation of the quasi-thermodynamic condition of $\dot V(x_t)\le 0$, where $V$ is given in \eqref{eq:V}, is not straightforward. Furthermore, the functional $V$ is not universal in the sense of \cite{gorban2019universal}, since it depends on the stoichiometric vectors and the rate coefficients. 
\begin{rem}
	It can be similarly shown that the largest invariant subset $\mathcal{M}$ of $\mathcal{R}\cap\mathcal{B}_{\epsilon}(\overline x)$, where $0<\epsilon<\min_{i=1,2,\dots,N}\overline x_i$, also consists of constant functions that are positive complex balanced equilibria. Since $\overline x$ is Lyapunov stable there exists $\delta>0$ such that $\theta\in\mathcal{B}_{\delta}(\overline x)$ implies that the corresponding solution $x_t^{\theta}\in\mathcal{B}_{\epsilon}(\overline x)$ for $t\ge0$. Applying LaSalle's invariance principle \cite{Smith2011} to $\mathcal{M}$ shows that for $\theta\in\mathcal{B}_{\delta}(\overline x)$ we have that $\omega(\theta)\subset\mathcal{M}\cap\mathcal{D}_{\theta}$, where $\omega(\theta)=\qty\big{\psi\in\overline{\mathcal{C}}_+\big|\text{there exists }t_n\rightarrow\infty\text{ such that }x_{t_n}^{\theta}\rightarrow\psi}$ is the omega limit set of $\theta$. But the elements of $\mathcal{M}$ are positive complex balanced equilibria and Proposition \ref{prop:dde:one} shows that $\mathcal{D}_{\theta}$ can contain at most one positive equilibrium; that is, we have that $\omega(\theta)$ consists of a single positive complex balanced equilibrium, which is also Lyapunov stable by Theorem \ref{thm:Lyapunov}. That is, the positive complex balanced equilibria of \eqref{eq:DDE} are semistable in the sense that they are not only Lyapunov stable but there exists $\delta>0$ such that for $\theta\in\mathcal{B}_{\delta}(\overline x)$ the corresponding solution $x^{\theta}(t)$ converges to a Lyapunov stable equilibrium.
\end{rem}

\begin{rem}
	Using the notations of \eqref{eq:ODE_matrix} system \eqref{eq:DDE} can be rewritten as
	\begin{equation}
		\dot x(t)=\sum_{i=1}^L\sum_{j=1}^L\kappa_{ij}\qty\big[\gamma^{y_i}\qty\big(x(t-\tau_{ij}))y_j-\gamma^{y_i}\qty\big(x(t))y_i].
	\end{equation}
	Then the Lyapunov-Krasovskii functional takes the form
	\begin{equation}
		\begin{aligned}
			V(\psi)&:=V(\psi,\overline x)=\sum_{i=1}^{N}\int_{\overline x_i}^{\psi_i(0)}\qty\big(\log\gamma_i(s)-\log\gamma_i(\overline x_i))\dd{s}\\
			&+\sum_{i=1}^L\sum_{j=1}^L\kappa_{ij}\int_{-\tau_{ij}}^0\qty\bigg(\gamma^{y_i}\qty\big(\psi(s))\qty\Big(\log\gamma^{y_i}\qty\big(\psi(s))-\log\gamma^{y_i}(\overline x)-1)-\gamma^{y_i}(\overline x))\dd{s}.
		\end{aligned}
	\end{equation}
	A calculation similar to the above shows that
	\begin{equation}
		\dot V(x_t)\le\sum_{i=1}^L\sum_{j=1}^L\kappa_{ij}\gamma^{y_i}(\overline x)\qty\big(e^{q_j(t)}-e^{q_i(t)}).
	\end{equation}
	The right-hand size is equal to 
	\begin{equation}
		\sum_{j=1}^Le^{q_j(t)}\qty\Bigg(\sum_{i=1}^L\kappa_{ij}\gamma^{y_i}(\overline x))-\sum_{i=1}^Le^{q_i(t)}\qty\Bigg(\sum_{j=1}^L\kappa_{ij})\gamma^{y_i}(\overline x)=Q(t)\tilde K\Gamma(\overline x).
	\end{equation}
	Since $\overline x$ is a complex balanced equilibrium, the vector $\Gamma(\overline x)$ is in the kernel of $\tilde K$; that is, we have that $\dot V(x_t)\le0$.
\end{rem}


\section{Examples}\label{sec:examples}
In the following examples we illustrate our notations and results.

\subsection{Example 1}
First, let us consider the delayed kinetic system from \cite{Liptak2018} with mass action kinetics. The system consists of a reversible reaction
\begin{equation}
	2X_1\xrightleftharpoons[\kappa_2=2,\tau_2=0.5]{\kappa_1=1}X_2.
\end{equation}
The corresponding kinetic system takes the form
\begin{equation}
	\dot x(t)=\kappa_1\qty\Bigg(x_1^2(t)\begin{bmatrix}0\\1\end{bmatrix}-x_1^2(t)\begin{bmatrix}2\\0\end{bmatrix})+\kappa_2\qty\Bigg(x_2(t-\tau_2)\begin{bmatrix}2\\0\end{bmatrix}-x_2(t)\begin{bmatrix}0\\1\end{bmatrix}).
\end{equation}
The stoichiometric subspace and its orthogonal complement is
\begin{equation}
	\mathcal{S}=\mathrm{span}\qty\Bigg{\begin{bmatrix}-1\\2\end{bmatrix}}\qquad\mathcal{S}^{\perp}=\mathrm{span}\qty\Bigg{\begin{bmatrix}2\\1\end{bmatrix}}.
\end{equation}
It is easy to verify that $[2~2]^{\top}$ is a positive complex balanced equilibrium, and thus the positive equilibria are given by
\begin{equation}
	\mathcal{E}=\qty\Bigg{x\in\mathbb{R}_+^2\Bigg|\begin{bmatrix}\log x_1-\log2\\\log x_2-\log2\end{bmatrix}\in\mathcal{S}^{\perp}}.
\end{equation}
For any $\overline x\in\mathcal{E}$ we consider the set of points
\begin{equation}
	\mathcal{X}_{\overline x}=\qty\Bigg{x\in\mathbb{R}_+^2\Bigg|\begin{bmatrix}x_1-\overline x_1\\(1+\kappa_2\tau_2)(x_2-\overline x_2)\end{bmatrix}\in\mathcal{S}}.
\end{equation}
If we construct constant functions in $\mathcal{C}_+$ from $\overline x$ and the elements of $\mathcal{X}_{\overline x}$ in the obvious way, then by \eqref{eq:Dt_char} we have $\mathcal{X}_{\overline x}\in\mathcal{D}_{\overline x}$.

Let us consider the transformations $\gamma_1(s)=\frac{s^2}{1+s}$ and $\gamma_2(s)=\frac{s^3}{1+s}$; that is, the transformed system takes the form
\begin{equation}
	\begin{aligned}
		\dot x(t)&=\kappa_1\qty\Bigg(\frac{x_1^4(t)}{\qty\big(1+x_1(t))^2}\begin{bmatrix}0\\1\end{bmatrix}-\frac{x_1^4(t)}{\qty\big(1+x_1(t))^2}\begin{bmatrix}2\\0\end{bmatrix})\\
			&+\kappa_2\qty\Bigg(\frac{x_2^3(t-\tau_2)}{1+x_2(t-\tau_2)}\begin{bmatrix}2\\0\end{bmatrix}-\frac{x_2^3(t)}{1+x_2(t)}\begin{bmatrix}0\\1\end{bmatrix}).
	\end{aligned}
\end{equation}
Is it easy to verify that $\qty\Big[\frac{\sqrt5}{2}+\frac{1}{2}~~1]^{\top}$ is a positive complex balanced equilibrium, and thus the positive equilibria are given by
\begin{equation}
	\mathcal{E}=\qty\Bigg{x\in\mathbb{R}_+^2\Bigg|\begin{bmatrix}\log\frac{x_1^2}{1+x_1}-\log1\\\log\frac{x_2^3}{1+x_2}-\log\frac{1}{2}\end{bmatrix}\in\mathcal{S}^{\perp}},
\end{equation}
and $\mathcal{X}_{\overline x}$ is given by
\begin{equation}
	\mathcal{X}_{\overline x}=\qty\Bigg{x\in\mathbb{R}_+^2\Bigg|\begin{bmatrix}x_1-\overline x_1\\x_2-\overline x_2+\kappa_2\tau_2\qty\Big(\frac{x_2^3}{1+x_2}-\frac{\overline x_2^3}{1+\overline x_2})\end{bmatrix}\in\mathcal{S}}.
\end{equation}
Using the terminology of \cite{Komatsu2019,Komatsu2020} it is easy to see that the set $W=\{X_1,X_2\}$ is the only minimal semilocking set (called siphon in the theory of Petri nets). The $L_W$ space consists of functions $w\in\overline{\mathcal{C}}_+$ such that
\begin{equation}
	\begin{aligned}
		w_i(s)&=0,\qquad X_i\in W,\\
		w_i(s)&\neq0,\qquad X_i\not\in W
	\end{aligned}
\end{equation}
holds for $s\in[-\tau,0]$. Then \cite[Theorem 5.1]{Komatsu2020} states that the boundary equilibria of the system is contained in
\begin{equation}
	\bigcup_{\theta\in\mathcal{C}_+}\overline{\mathcal{D}}_{\theta}\cap L_W,
\end{equation}
but the above set consists of only the constant zero function; that is, all nontrivial equilibria are positive and globally asymptotically stable w.r.t. their positive stoichiometric compatibility classes.

In Figure \ref{fig:ex1:phase}, the positive equilibria, several positive stoichiometric compatibility classes and trajectories of the original mass action system are depicted with red dashed, green dashed and green continuous lines, respectively. The same objects for the transformed system are drawn with black dashed, blue dashed and blue continuous lines, respectively.

\begin{figure}[H]
	\begin{center}
		\includegraphics[width=0.8\textwidth]{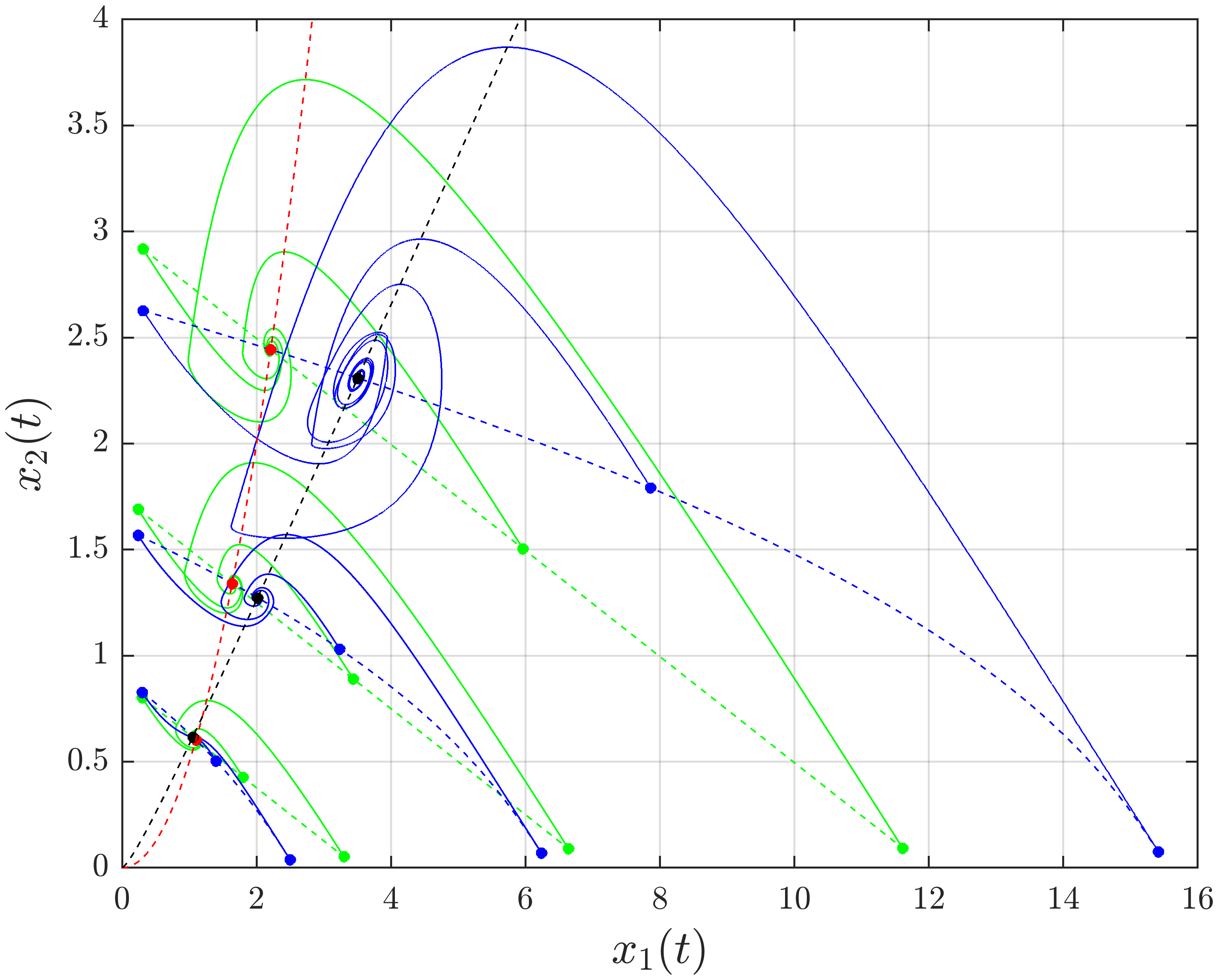}
	\end{center}
	\caption{Phase plot of Example 1}\label{fig:ex1:phase}
\end{figure}

\subsection{Example 2}
Our next example is a delayed version of another complex balanced small reaction network, taken from \cite{Szederkenyi2011}. We consider the set of reversible reactions
\begin{equation}
	3X_1\xrightleftharpoons[\kappa_2=\frac{2.8}{3}]{\kappa_1=\frac{1.4}{3}}3X_2\qquad3X_1\xrightleftharpoons[\kappa_4=0.126,\tau_4=0.4]{\kappa_3=0.1}2X_1+X_2\qquad3X_2\xrightleftharpoons[\kappa_6=0.063,\tau_6=0.6]{\kappa_5=0.1}2X_1+X_2
\end{equation}
with the transformations $\gamma_1(s)=s$ and $\gamma_2(s)=\frac{s^2}{1+s}$. Then the system takes the form
\begin{equation}
	\begin{aligned}
		\dot x(t)&=\kappa_1\qty\Bigg(x_1^3(t)\begin{bmatrix}0\\3\end{bmatrix}-x_1^3(t)\begin{bmatrix}3\\0\end{bmatrix})+\kappa_2\qty\Bigg(\frac{x_2^6(t)}{\qty\big(1+x_2(t))^2}\begin{bmatrix}3\\0\end{bmatrix}-\frac{x_2^6(t)}{\qty\big(1+x_2(t))^2}\begin{bmatrix}0\\3\end{bmatrix})\\
			&+\kappa_3\qty\Bigg(x_1^3(t)\begin{bmatrix}2\\1\end{bmatrix}-x_1^3(t)\begin{bmatrix}3\\0\end{bmatrix})\\
				&+\kappa_4\qty\Bigg(x_1^2(t-\tau_4)\frac{x_2^2(t-\tau_4)}{1+x_2(t-\tau_4)}\begin{bmatrix}3\\0\end{bmatrix}-x_1^2(t)\frac{x_2^2(t)}{1+x_2(t)}\begin{bmatrix}2\\1\end{bmatrix})\\
					&+\kappa_5\qty\Bigg(\frac{x_2^6(t)}{\qty\big(1+x_2(t))^2}\begin{bmatrix}2\\1\end{bmatrix}-\frac{x_2^6(t)}{\qty\big(1+x_2(t))^2}\begin{bmatrix}0\\3\end{bmatrix})\\
						&+\kappa_6\qty\Bigg(x_1^2(t-\tau_6)\frac{x_2^2(t-\tau_6)}{1+x_2(t-\tau_6)}\begin{bmatrix}0\\3\end{bmatrix}-x_1^2(t)\frac{x_2^2(t)}{1+x_2(t)}\begin{bmatrix}2\\1\end{bmatrix}).
	\end{aligned}
\end{equation}
The stoichiometric subspace and its orthogonal complement are
\begin{equation}
	\mathcal{S}=\mathrm{span}\qty\Bigg{\begin{bmatrix}-3\\3\end{bmatrix}}\qquad\mathcal{S}^{\perp}=\mathrm{span}\qty\Bigg{\begin{bmatrix}3\\3\end{bmatrix}}.
\end{equation}
It is easy to verify via the Cardano formula that 
\begin{equation}
	\overline x=\begin{bmatrix}\sqrt[3]2\\\sqrt[3]{\frac{1}{2}+\sqrt{\frac{23}{108}}}+\sqrt[3]{\frac{1}{2}-\sqrt{\frac{23}{108}}}\end{bmatrix}
\end{equation}
is a positive complex balanced equilibrium, and thus the positive equilibria are given by
\begin{equation}
	\mathcal{E}=\qty\Bigg{x\in\mathbb{R}_+^2\Bigg|\begin{bmatrix}\log x_1-\log\overline x_1\\\log\frac{x_2^2}{1+x_2}-\log\frac{\overline x_2^2}{1+\overline x_2}\end{bmatrix}\in\mathcal{S}^{\perp}},
\end{equation}
and $\mathcal{X}_{\overline x}$ is given by
\begin{equation}
	\mathcal{X}_{\overline x}=\qty\Bigg{x\in\mathbb{R}_+^2\Bigg|\begin{bmatrix}x_1-\overline x_1+2(\kappa_4\tau_4+\kappa_5\tau_5)\qty\Big(x_1^2\frac{x_2^2}{1+x_2}-\overline x_1^2\frac{\overline x_2^2}{1+\overline x_2})\\x_2-\overline x_2+(\kappa_4\tau_4+\kappa_5\tau_5)\qty\Big(x_1^2\frac{x_2^2}{1+x_2}-\overline x_1^2\frac{\overline x_2^2}{1+\overline x_2})\end{bmatrix}\in\mathcal{S}}.
\end{equation}
Similarly to the previous example, it can be shown via \cite[Theorem 5.1]{Komatsu2020} that all nontrivial equilibria of the system are positive and globally asymptotically stable w.r.t. their positive stoichiometric compatibility classes.

In Figure \ref{fig:ex2:phase}, the positive equilibria, several positive stoichiometric compatibility classes and trajectories of system are drawn with black dashed, blue dashed and blue continuous lines, respectively.

\begin{figure}[H]
	\begin{center}
		\includegraphics[width=0.8\textwidth]{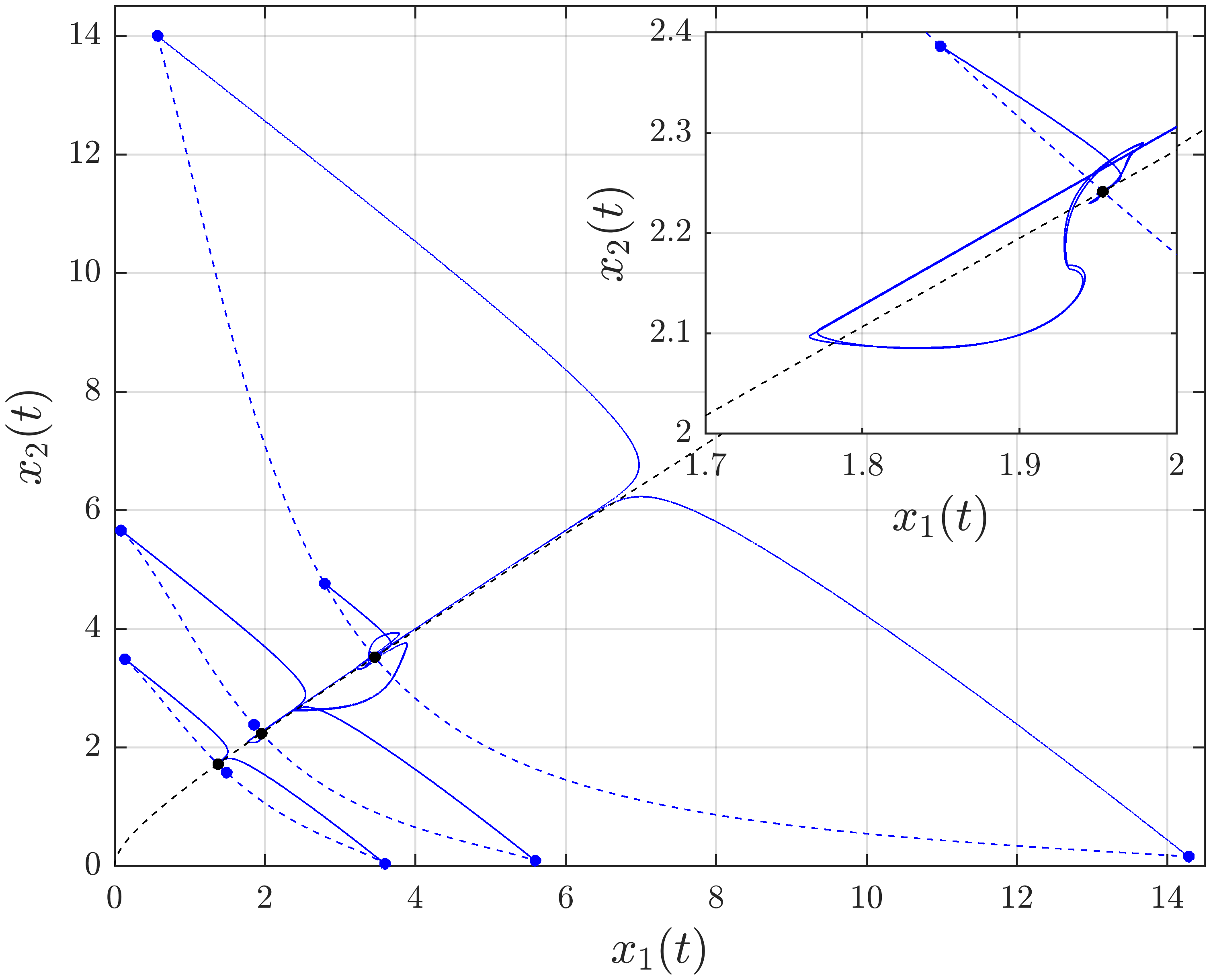}
	\end{center}
	\caption{Phase plot of Example 2}\label{fig:ex2:phase}
\end{figure}

\subsection{Example 3}
Our final example focuses on the Lyapunov-Krasovskii functional. Of course it cannot be visualized in general as it maps an infinite dimensional function space to nonnegative numbers. However, if we restrict the functional to constant history functions as in the previous examples, then we can compare it to the nondelayed Lyapunov function. In order to do so, we consider the following delayed reversible reactions
\begin{equation}
	2X_1\xrightleftharpoons[\kappa_2=1]{\kappa_1=1,\tau_1=1}2X_3\qquad2X_1+X_2\xrightleftharpoons[\kappa_4=2,\tau_4=0.5]{\kappa_3=1}3X_3,
\end{equation}
with transformations $\gamma_1(s)=s$, $\gamma_2(s)=\frac{s^2}{1+s}$ and $\gamma_3(s)=\frac{s}{1+s}$. Omitting the vector notation, the corresponding delayed differential equation takes the form
\begin{equation}
	\begin{aligned}
		\dot x_1(t)&=-2\kappa_1x_1^2(t)+2\kappa_2\qty\bigg(\frac{x_3(t)}{1+x_3(t)})^2-2\kappa_3x_1^2(t)\frac{x_2^2(t)}{1+x_2(t)}+2\kappa_4\qty\bigg(\frac{x_3(t-\tau_4)}{1+x_3(t-\tau_4)})^3\\
		\dot x_2(t)&=\kappa_4\qty\bigg(\frac{x_3(t-\tau_4)}{1+x_3(t-\tau_4)})^3-\kappa_3x_1^2(t)\frac{x_2^2(t)}{1+x_2(t)}\\
		\dot x_3(t)&=2\kappa_1x_1^2(t-\tau_1)-2\kappa_2\qty\bigg(\frac{x_3(t)}{1+x_3(t)})^2+3\kappa_3x_1^2(t)\frac{x_2^2(t)}{1+x_2(t)}-3\kappa_4\qty\bigg(\frac{x_3(t)}{1+x_3(t)})^3.
	\end{aligned}
\end{equation}
It is easy to see that the nondelayed system is conservative as $x_1+x_2+x_3$ is a first integral; that is, the nondelayed positive stoichiometric compatibility classes can be characterized as
\begin{equation}
	\mathcal{S}_p=\qty\big{x\in\mathbb{R}_+^3\big|x_1+x_2+x_3=p_1+p_2+p_3},
\end{equation}
where $p\in\mathbb{R}_+^3$ is arbitrary. Then for any fixed $p\in\mathbb{R}_+^3$ we can visualize the Lyapunov function \eqref{eq:gen:Lyapunov} as a two-dimensional function defined on the region
\begin{equation}
	\mathcal{D}_p=\qty\big{x\in\mathbb{R}_+^2\big|x_1+x_2\le p_1+p_2+p_3}.
\end{equation}
The delayed positive stoichiometric compatibility class of the delayed system is more complicated and, in particular, it is not a plane; that is, the delayed system is not conservative in this sense. However, it can be shown similarly to the previous examples that the system is persistent, and thus every delayed positive stoichiometric compatibility class contains precisely one positive equilibrium. Assuming a constant history function constructed from an element of $\mathcal{D}_p$, we can compute the value of the functional at the initial point of the corresponding trajectory. Figure \ref{fig:ex3} shows the contour plots of the Lyapunov function and the Lyapunov-Krasovskii functional on $\mathcal{D}_p$ with $p_1+p_2+p_3=1$.

\begin{figure}[H]
	\begin{subfigure}[b]{0.49\textwidth}
		\centering
		\includegraphics[width=\textwidth]{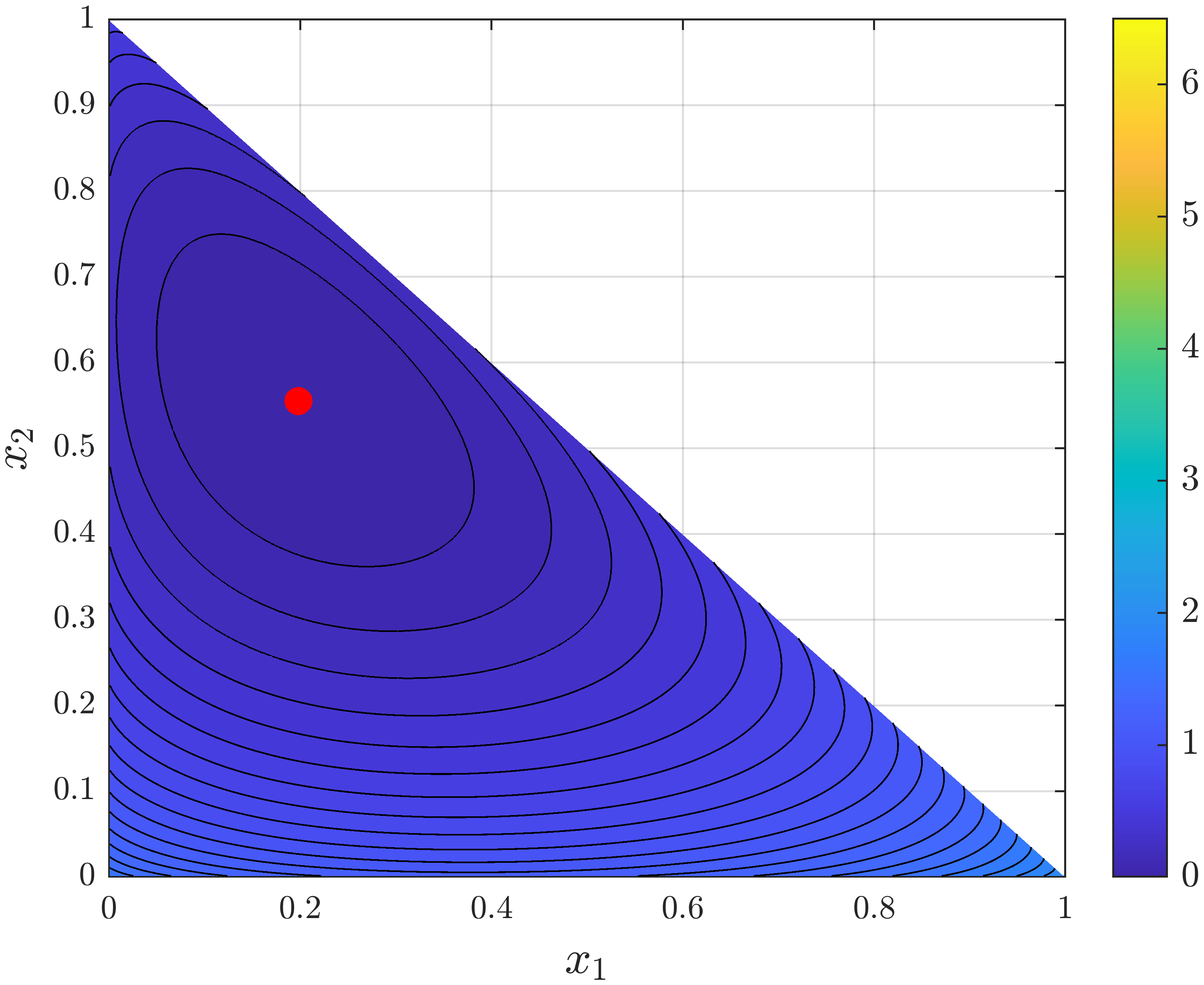}
		\caption{Lyapunov function}
	\end{subfigure}
	\hfill
	\begin{subfigure}[b]{0.49\textwidth}
		\centering
		\includegraphics[width=\textwidth]{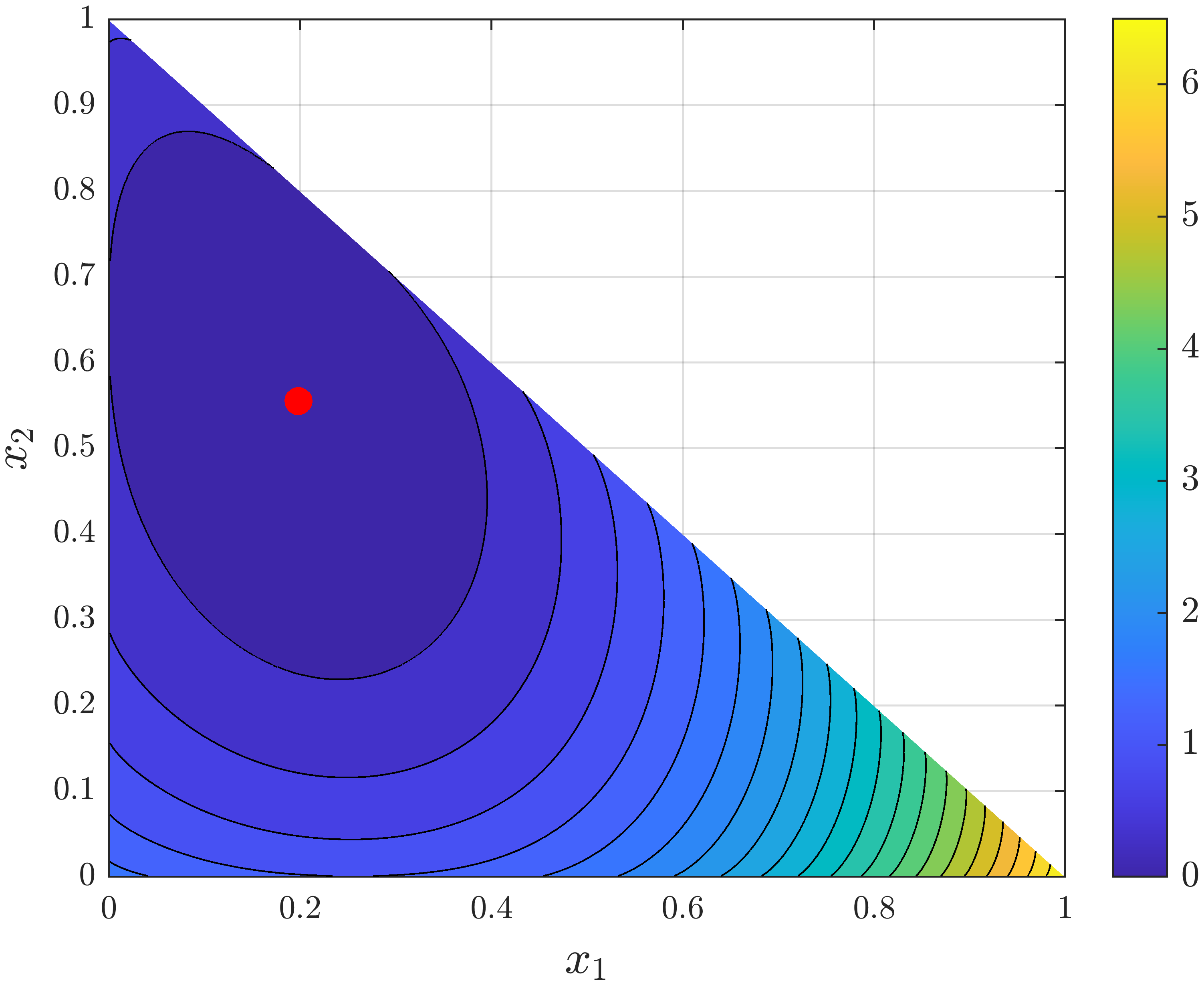}
		\caption{Lyapunov-Krasovskii functional}
	\end{subfigure}
	\caption{Level curves of the Lyapunov function of the nondelayed system and the Lyapunov-Krasovskii functional of the delayed system for constant history functions}\label{fig:ex3}
\end{figure}

\section{Conclusions}\label{sec:concl}
The stability of kinetic systems with time delays and general kinetics was studied in this paper. In preparation for the subsequent analysis, certain stability results of \cite{Sontag2001} were slightly generalized using the notion of quasi-thermodynamicity introduced in \cite{Horn1972}. Then it was shown for delayed complex balanced reaction networks that each positive stoichiometric compatibility class contains precisely one positive equilibrium that is locally asymptotically stable within their positive stoichiometric compatibility classes for arbitrary finite time delays. A key result of the paper allowing the stability proof is the construction of an appropriate Lyapunov-Krasovskii functional. Thus, the results proposed in \cite{Liptak2018} have been generalized for a wide class of delayed non-mass action reaction networks. It was also shown that the global stability of equilibria can be proved as well if the conditions in \cite{Komatsu2019,Komatsu2020} are fulfilled. Three illustrative examples were given to visualize the theoretical results. Further work will be focused on the kinetic realization and control of general nonlinear delayed models given in DDE form.

\section{Acknowledgements}
The authors acknowledge the support of the Hungarian National Research, Development and Innovation Office (NKFIH) through the grant 145934 and of the ÚNKP-23-3-II-PPKE-81 National Excellence Program of the Ministry for Culture and Innovation from the source of the National Research, Development and Innovation Fund.


\end{document}